\documentclass[12pt]{amsart}
\usepackage{amsmath,amsfonts,amssymb}
\numberwithin{equation}{section}

\newtheorem{theorem}{Theorem}[section]
\newtheorem{lemma}[theorem]{Lemma}

\newtheorem{proposition}[theorem]{Proposition}

\newtheorem{definition}[theorem]{Definition}

 \DeclareMathOperator{\Tr}{Tr}
\DeclareMathOperator{\Ad}{Ad} 
\DeclareMathOperator{\id}{id}

\DeclareMathOperator{\End}{End} \DeclareMathOperator{\Rea}{Re}
\DeclareMathOperator{\Ima}{Im}

\title [Weyl calculus on the Fock space...]{ Weyl calculus on the Fock space and Stratonovich-Weyl correspondence for Heisenberg motion groups}
\author {Benjamin Cahen}
\address{Universit\'e de Lorraine, Site de Metz, UFR-MIM,
D\'epartement de math\'ematiques,
B\^atiment A,
3 rue Augustin Fresnel, BP 45112,
57073 METZ Cedex 03, France.}
\email{benjamin.cahen@univ-lorraine.fr}

\subjclass[2000]{22E45; 22E70; 22E20; 81S10; 81R30.} \keywords{Weyl calculus;
Stratonovich-Weyl correspondence; Berezin quantization; Heisenberg group; Heisenberg motion group; 
reproducing kernel Hilbert space; coherent states; 
Bargmann-Fock representation.}

\begin{document}

\maketitle

\begin{abstract}
We construct a Stratonovich-Weyl correspondence for each generic representation of a Heisenberg motion group by using the Weyl calculus on the Fock space.
\end{abstract}

\vspace{1cm}

\section {Introduction} \label{sec:intro}

The notion of Stratonovich-Weyl correspondence was first introduced in
\cite{St} as a generalization of the classical Weyl correspondence between functions on  ${\mathbb R}^{2n}$ and operators
on $L^2( {\mathbb R}^{n})$ to the setting of a Lie group acting on a homogeneous space. 
The systematic study of Stratonovich-Weyl
correspondences  began with the work of J.M. Gracia-Bond\`{i}a, J.C.
V\`{a}rilly and their co-workers, see in particular \cite{GBV}, \cite{FGBV},
\cite{CGBV} and \cite{BM}. A good review of  Stratonovich-Weyl correspondences and
related topics can be found in \cite{GB}. 

 The following definition is taken from \cite{GB}, see also \cite{GBV}.

\begin{definition} Let $G$ be a Lie group and
$\pi$ be a unitary representation of $G$ on a Hilbert space $\mathcal
H$. Let $M$ be a homogeneous $G$-space and let $\mu$ be
a $G$-invariant measure on $M$. Then a
Stratonovich-Weyl correspondence for the triple $(G,\pi, M)$ is an
isomorphism ${\mathcal W}$ from a vector space of operators on $\mathcal H$ to
a vector space of functions on $M$ satisfying the following
properties:

\begin{enumerate}

\item the function ${\mathcal W}(A^{\ast})$ is the complex-conjugate of ${\mathcal W}(A)$;

\item Covariance: we have ${\mathcal W}(\pi (g)\,A\,\pi (g)^{-1})(x)={\mathcal W}(A)(g^{-1}\cdot x)$;

\item Traciality: we have
\begin{equation*}\int_M\,{\mathcal W}(A)(x){\mathcal W}(B)(x)\,d\mu (x)=\Tr(AB).
\end{equation*}

\end{enumerate} \end{definition}

Note that Stratonovich-Weyl
correspondences are particular cases of the invariant symbolic calculi  introduced
and studied by J. Arazy and H. Upmeir, especially for symmetric domains, see \cite{AU1}, \cite{AU2}, \cite{AU3}.

In the previous definition, $M$ is generally taken to be a coadjoint orbit of $G$ which is associated with $\pi$ by the Kirillov-Kostant
method of orbits \cite{Kir}, \cite{Ko}. A basic example is then the case when $G$ is the $(2n+1)$-dimensional Heisenberg group. Each non-degenerate coadjoint orbit $M$ of $G$ is diffeomorphic to ${\mathbb R}^{2n}$ and is associated with  a unitary irreducible representation $\pi$  of $G$ on $L^2({\mathbb R^n})$. In this case,
it is well-known that the classical Weyl correspondence gives a Stratonovich-Weyl
correspondence for the triple $(G,\pi ,M)$
\cite{Fo}, \cite{GB}. More generally, the Pedersen-Weyl calculus on the flat coadjoint orbits of a nilpotent Lie group $G$ also provides a Stratonovich-Weyl
correspondence for the unitay irreducible representations of $G$ which are square-integrable modulo the center of $G$ \cite{Pe}, \cite{BB}.

On the other hand, Stratonovich-Weyl correspondences for the holomorphic
representations of Hermitian Lie groups were obtained by
taking the isometric part in the polar decomposition of the Berezin
quantization map, see \cite{CaPad}, \cite{CaSWC}, \cite{CaSWD}, \cite{FGBV}.

In \cite{CaRimut2}, we considered the case where $G$ is a Heisenberg motion group as a non-trivial example beyond the nilpotent and Hermitian cases. 
A Heisenberg motion group is the semidirect product of the $(2n+1)$-dimensional Heisenberg group $H_n$ by a compact subgroup $K$ of the unitary group $U(n)$. Note that the Heisenberg motion groups play an important role in the theory of Gelfand pairs \cite{BJLR}, \cite{BJR}. 

Each generic unitary irreducible representation $\pi$ of $G=H_n\rtimes K$ is holomorphically induced from the tensor product of a character of the center of $H_n$ by a unitary irreducible representation $\rho$ of $K$  \cite{BJLR}. Then $\pi$ can be associated, in the Kirillov-Kostant method of orbits, 
with a coadjoint $\mathcal O$ of $G$ which is diffeomorphic to ${\mathbb C}^n\times o$ where $o$ is the coadjoint orbit of $K$ associated with $\rho$.

In \cite{CaRimut2}, we used the Bargmann transform in order to get a 
Schr\"odinger realization of $\pi$ in a space of square-integrable functions
on ${\mathbb R}^{2n}$ and we showed that the usual Weyl correspondence
then gives a Stratonovich-Weyl correspondence for $\pi$. Note that, in \cite{CaBW}, we used a similar method to get an adapted Weyl correspondence in the sense of \cite{CaWQ}.  However, in the 
Schr\"odinger realization, it is difficult to obtain explicit formulas for the representation operators.

So, we propose here a slightly different method based on the Bargmann-Fock version of the Weyl calculus which appears as a particular case of the Weyl calculus on symmetric domains \cite{AU1}, \cite{AU3}. Combining the Weyl calculus on ${\mathbb C}^n$ with a Stratonovich-Weyl correspondence on $o$ \cite{FGBV}, \cite{CaSWC}, we obtain a Stratonovich-Weyl correspondence for $\pi$. In particular, we exhibit the
Stratonovich-Weyl quantizer and show that the Stratonovich-Weyl correspondence
can also be obtained by using the general method of \cite{GB} and \cite{AMO}.

The paper is organized as follows. First, we recall the Berezin calculus (Section \ref{sec:2}) and the complex Weyl calculus (Section \ref{sec:3}) for the (non-degenerate) unitary irreducible representations of the Heisenberg group. In Section \ref{sec:4} and \ref{sec:5}, we introduce the Heisenberg motion groups, their generic representations and the associated Berezin calculus. The Stratonovich-Weyl correspondence for a generic representation
$\pi$ is constructed in Section \ref{sec:6}. In particular, we compute
the Stratonovich-Weyl symbol of $d\pi(X)$ for each $X$ in the Lie algebra of $G$.

\section{Heisenberg group: Berezin quantization} \label{sec:2}

In this section, we review some well-known results about the
the Bargmann-Fock
model for the unitary irreducible (non-degenerated)
representations of the Heisenberg group. We follow the presentation
of \cite{CaRiv} and \cite{CaRimut2}.

For each $z,\,w \in {\mathbb C}^{n}$, we denote $zw:=\sum_{k=1}^nz_kw_k$. For each $z, z'\in {\mathbb C}^{n}$, let us define
\begin{equation*}\omega
(z,z')=\frac{i}{2}(z{\bar z'}-{\bar z}z').
\end{equation*}
Then the Heisenberg group $G_0$ of dimension $2n+1$ is ${\mathbb C}^{n}\times {\mathbb R}$ equipped with the multiplication
\begin{equation*}(z,c)\cdot (z',c')=(z+z',c+c'+\tfrac{1}{2}\omega (z,z')).\end{equation*}

The Lie algebra ${\mathfrak g}_0$ of $G_0$ can be identified to
${\mathbb C}^{n}\times {\mathbb R}$ with the Lie brackets
\begin{equation*}[(v,c), (v',c')]=(0,\omega (v,v')).\end{equation*}

Let $(e_k)_{1\leq k\leq n}$ be the canonical basis 
of ${\mathbb C}^{n}$ and let $\{X_1,\ldots,X_n,Y_1,\ldots,Y_n,Z\}$ be the basis of ${\mathfrak g}_0$ defined by $X_k=(e_k,0), Y_k=(ie_k,0)$ for $k=1,2, \ldots,n$ and $Z=(0,1)$. Then the only non trivial brackets in this basis are
$[X_k\,,\,Y_k]=Z$ for $k=1,2, \ldots,n$. In particular, we have, for each $(v,c)\in {\mathfrak g}_0$,
\begin{equation*}(v,c)=\sum_{k=1}^n(\Rea v_k)X_k+(\Ima v_k)Y_k+
cZ.\end{equation*}

Also, we denote by $\{X_1^{\ast},\ldots,X_n^{\ast},Y_1^{\ast},\ldots,Y_n^{\ast},{
Z}^{\ast}\}$ be the corresponding dual basis of ${\mathfrak
g}_0^{\ast}$. For each $\alpha \in {\mathbb C}^{n}$ and $\gamma \in {\mathbb R}$, let $(\alpha, \gamma)_{\ast}$ denote the element of ${\mathfrak
g}_0^{\ast}$ defined by 
\begin{equation*}\langle (\alpha, \gamma)_{\ast}, (v,c)\rangle := \omega (\alpha,  v)+\gamma c \end{equation*} for each $(v,c)\in {\mathfrak g}_0$. Then we have
\begin{equation*}(\alpha, \gamma)_{\ast}=
\sum_{k=1}^n(-\Ima {\alpha}_k)X_k^{\ast}+(\Rea {\alpha}_k)Y_k^{\ast}+
{\gamma}Z^{\ast}.\end{equation*}

For each real number $\lambda>0$, we denote by ${\mathcal O}_{\lambda}$ the
orbit of the element $(0,\lambda)_{\ast}=\lambda Z^{\ast}$ of
${\mathfrak g}_0^{\ast}$ under the coadjoint action of $G_0$ (the
case $\lambda<0$ can be treated similarly). By the Stone-von Neumann
theorem, there exists a unique (up to unitary equivalence) unitary
irreducible representation of $G_0$ whose restriction to the center
of $G_0$ is the character $(0,c)\rightarrow e^{i\lambda c}$
\cite{AKo}, \cite{Fo}. In fact, this representation is associated with the
coadjoint orbit ${\mathcal O}_{\lambda}$ by the Kirillov-Kostant
method of orbits \cite{Kir}, \cite{Ko}. Indeed, if we take the polarization at $\lambda Z^{\ast}$ to be the space spanned
by the elements $X_k+iY_k$ for $ k=1,2,\ldots, n$ and $Z$
then the method of orbits leads to the Bargmann-Fock representation $\pi_0$
defined as follows. Note that, for our method, we take here a complex polarization at $\lambda Z^{\ast}$ instead of a real one.

Let ${\mathcal H}_{0}$ be the Hilbert space of holomorphic functions $f$
on ${\mathbb C}^n$ such that \begin{equation*}\Vert f\Vert^2_{{\mathcal H}_{0}}
:=\int_{{\mathbb C}^n} \vert f(z)\vert^2\, e^{-\vert
z\vert^2/2\lambda}\,d\mu_{\lambda} (z) <+\infty\end{equation*} where
 $d\mu_{\lambda}(z):=(2\pi
\lambda)^{-n}\,dx\,dy$. Here $z=x+iy$ with $x$ and $y$ in ${\mathbb
R}^n$.

Then we have

\begin{equation*}({\pi}_0(g)f)(z)=\exp \left(i\lambda c_0+\tfrac{1}{2}i{\bar z_0}z-\tfrac{\lambda}{4}\vert z_0\vert^2\right)\,f(z+i\lambda z_0) \end{equation*}
for each $g=(z_0, c_0)\in G$ and $z\in {\mathbb C}^n$.

The differential of $\pi_0$ is thus given by

\begin{equation*}(d\pi_0 (v,c)f)(z)=i(\lambda c+\tfrac{1}{2}{\bar v}z)f(z)+df_z(i\lambda v)\end{equation*}
or, equivalently, by

\begin{equation*}\left\{\begin{aligned}
d\pi_{0}(X_k)f(z)=&\frac{1}{2}iz_kf(z)+\lambda i\frac
{\partial f}{\partial z_k}\\
d\pi_{0}(Y_k)f(z)=&\frac{1}{2}z_kf(z)-\lambda \frac {\partial
f}{\partial z_k}\\
d\pi_{0}(Z)f(z)=&i\lambda f(z).
\end{aligned}\right.\end{equation*}

Now, we recall the definition and some of the properties of the Berezin
quantization map \cite{Be1}, \cite{Be2}, \cite{CaRiv}.

For each $z\in {\mathbb C}^n$, consider the coherent state
$e_z(w)=\exp ({\bar z}w/2\lambda)$. Then we have the reproducing property
$f(z)=\langle f,e_z\rangle_{{\mathcal H}_{0}}$ for each $f\in {\mathcal
H}_{0}$.

Let ${\mathcal C}_0$ be the space of all operators (not necessarily
bounded) $A$ on ${\mathcal H}_{0}$ whose domain contains $e_z$
for each $z\in {\mathbb C}^n$. Then the Berezin symbol of $A\in {\mathcal C}_0$ is the
function $S_0(A)$ defined on ${\mathbb C}^n$ by
\begin{equation*}S_0(A)(z):=\frac {\langle A\,e_z\,,\,e_z\rangle_{\mathcal H_0}}{\langle e_z\,,\,e_z\rangle_{\mathcal H_0}}. \end{equation*}

Let us consider now the action of $G_0$ on ${\mathbb C}^n$ defined by
$g\cdot z =z-i\lambda z_0$ for $g=(z_0,c_0)$. 
We have the following result, see for instance \cite{CaRiv}.

\begin{proposition} \label{prop:Ber}\begin{enumerate}
\item Each $A\in {\mathcal C}_0$ is determined by $S_0(A)$;

\item For each $A\in {\mathcal C}_0$ and each $z\in {\mathbb C}^n$, we
have $S_0(A^{\ast})(z)=\overline {S_0(A)(z)}$;
 
\item For each $z\in {\mathbb C}^n$, we have $S_0(\id_{{\mathcal H}_{0}})(z)=1$. Here $\id_{{\mathcal H}_{0}}$ denotes the identity operator of ${\mathcal H}_0$;

\item For each $A\in {\mathcal C}_0$, $g\in G_0$ and $z\in {\mathbb C}^n$, we have
$\pi_0(g)^{-1}A\pi_0(g)\in {\mathcal C}_0$ and
\begin{equation*}S_0(A)(g\cdot z)=S_0(\pi_0(g)^{-1}A\pi_0(g))(z);\end{equation*}

\item The map $S_0$ is a bounded operator from
${\mathcal L}_2({\mathcal H}_{0})$ (endowed with the
Hilbert-Schmidt norm) to $L^2({\mathbb C}^n,\mu_{\lambda})$  which
is one-to-one and has dense range.

\end{enumerate} \end{proposition}

The map $\Phi_{\lambda}:{\mathbb C}^n\rightarrow {\mathcal O}_{\lambda}, z\rightarrow (-iz,\lambda)_{\ast}$
is clearly a diffeomorphism satisfying
\begin{equation*}
\Phi_{\lambda}(g\cdot z)=\Ad^{\ast}(g)\,\Phi_{\lambda}(z) \end{equation*}
for each $g\in G_0$ and $z\in {\mathbb C}^n$. Moreover, we have,
for each $X\in {\mathfrak g}_0$ and each
$z\in {\mathbb C}^n$,
\begin{equation*}S_0(d\pi_{0}(X))(z)=i \langle \Phi_{\lambda}
(z),X\rangle. \end{equation*}
This property gives a connection between $\pi_0$ and ${\mathcal O}_{\lambda}$.

Note that the map $z\rightarrow g_z:=({\lambda}^{-1}iz,0)$ is a section
for the action of $G$ on ${\mathbb C}^n$, that is, we have $g_z \cdot 0=z$
for each $z\in {\mathbb C}^n$.

\section{Heisenberg group: Weyl quantization} \label{sec:3}

Here we apply the general method for constructing Stratonovich-Weyl correspondence \cite{GB}, \cite{AMO} and then recover the Bargmann-Fock version of  the Weyl
calculus, see \cite{AU1}, Example 2.2 and Example 4.2.

We start from the so-called Stratonovich-Weyl quantizer $\Omega$. Here
it is generated by the parity operator $R_0$ of ${\mathcal H}_0$ defined by
\begin{equation*}(R_0f)(z)=2^nf(-z). \end{equation*}
More precisely, we define
\begin{equation*}\Omega_0(z):=\pi_0(g_z)R_0\pi_0(g_z)^{-1} \end{equation*} for each $z\in {\mathbb C}^n$. Then we get immediately
\begin{equation}\label{eq:omega0}(\Omega_0(z)f)(w)=2^n\exp \left(\tfrac{1}{\lambda}(w{\bar z}-\vert z\vert^2)\right) f(2z-w) \end{equation}
for each $z, w\in {\mathbb C}^n$. Thus $\Omega_0$ satisfies the covariance
property
\begin{equation}\label{eq:Om0}\Omega_0(g\cdot z)=\pi_0(g)\Omega_0(z)\pi_0(g)^{-1} \end{equation} for each $g\in G_0$ and $z \in {\mathbb C}^n$.

We are now in position to recover the complex Weyl calculus. For each trace-class operator $A$ on ${\mathcal H}_0$, let $W_0(A)$ be the function on
${\mathbb C}^n$ defined by
\begin{equation*} W_0(A)(z):=\Tr(A\Omega_0(z))\end{equation*}
for each $z \in {\mathbb C}^n$.

We can give an integral expression for $W_0(A)$ as follows. 
For each trace class operator $A$ on ${\mathcal H}_0$, let $k_A(z,w)$ be the kernel of $A$, that is, for each $f \in {\mathcal H}_0$ and $z \in {\mathbb C}^n$, we have
\begin{equation*}(Af)(z)=
\int_{{\mathbb C}^n}k_A(z,w)f(w)\,e^{-\vert
w\vert^2/2\lambda}d\mu_{\lambda}(w).
\end{equation*} 

Note that, by the reproducing property, we have, for each $f \in {\mathcal H}_0$ and $z \in {\mathbb C}^n$
\begin{align*} (Af)(z)=&{\langle Af,e_z\rangle_{\mathcal H_0}}
={\langle f,A^{\ast}e_z\rangle_{\mathcal H_0}}\\
=&\int_{{\mathbb C}^n}f(w)\overline{(A^{\ast}e_z)(w)}\,e^{-\vert
w\vert^2/2\lambda}d\mu_{\lambda}(w) \end{align*}
so that we get
\begin{equation*}k_A(z,w)=\overline{(A^{\ast}e_z)(w)}=
\overline{\langle A^{\ast}e_z, e_w\rangle_{\mathcal H_0}}
=\langle Ae_w,e_z\rangle_{\mathcal H_0}\end{equation*}
for each $z, w \in {\mathbb C}^n$. This shows that $k_A(z,w)$ is
holomorphic in $z$ and anti-holomorphic in $w$. 

\begin{proposition} \label{prop:intW0}For each trace-class operator $A$ on ${\mathcal H_0}$
and each $z \in {\mathbb C}^n$, we have
\begin{equation*}
W_0(A)(z)=2^n\int_{{\mathbb C}^n}k_A(w,2z-w)\exp \left(\tfrac{1}{\lambda}\left(-z{\bar z}+z{\bar w}-\tfrac{1}{2}w{\bar w}\right)\right) d\mu_{\lambda}(w). 
\end{equation*}\end{proposition}

\begin{proof} By Eq. \ref{eq:omega0}, we have for each $f \in {\mathcal H}_0$ and $z,z'\in {\mathbb C}^n$
\begin{align*}
(A\Omega_0(z)&f)(z')=\int_{{\mathbb C}^n}k_A(z',w)(\Omega_0(z)f)(w)
e^{-\vert
w\vert^2/2\lambda}d\mu_{\lambda}(w)\\
=&2^n\int_{{\mathbb C}^n}k_A(z',w)f(2z-w)
\exp \left(\tfrac{1}{\lambda}\left({\bar z}w-z{\bar z}-\tfrac{1}{2}w{\bar w}\right)\right) d\mu_{\lambda}(w)\\
=&2^n\int_{{\mathbb C}^n}k_A(z',2z-w)f(w)
\exp \left(\tfrac{1}{\lambda}\left(-z{\bar z}+z{\bar w}\right)\right) 
\,e^{-\vert
w\vert^2/2\lambda}d\mu_{\lambda}(w).
\end{align*}
This shows that the kernel of $A\Omega_0(z)$ is
\begin{equation*}k_{A\Omega_0(z)}(z',w)=2^n k_A(z',2z-w)
\exp \left(\tfrac{1}{\lambda}\left(-z{\bar z}+z{\bar w}\right)\right). 
\end{equation*}
Then, by applying Mercer's theorem, we get
\begin{align*}
W_0(A)&(z)=\int_{{\mathbb C}^n}k_{A\Omega_0(z)}(w,w)e^{-\vert
w\vert^2/2\lambda}d\mu_{\lambda}(w)\\
&=2^n \int_{{\mathbb C}^n}k_A(w,2z-w)\exp \left(\tfrac{1}{\lambda}\left(-z{\bar z}+z{\bar w}\right)\right)e^{-\vert
w\vert^2/2\lambda}d\mu_{\lambda}(w). \end{align*}
\end{proof}

Let us recall the usual Weyl correspondence and briefly describe its connection with $W_0$.
Let $G_0$ act on ${\mathbb R}^{2n}$ by
\begin{equation*}(z_0,c_0)\cdot (p,q):=(p+\Rea z_0,q+\lambda \Ima z_0).
\end{equation*}
This action corresponds to the preceding action of $G_0$ on ${\mathbb C}^n$ via the identification ${\mathbb R}^{2n}\simeq {\mathbb C}^{n}$
given by $(p,q)\leftrightarrow z=q-i\lambda p$.

Let $R_1$ be the operator on $L^2({\mathbb R}^{n})$ defined by
\begin{equation*}(R_1\phi)(x)=2^n\phi(-x)
\end{equation*}
and let $\pi_1$ be the 
Schr\"odinger representation of $G_0$ on
$L^2({\mathbb R}^n)$ is defined by
\begin{equation*}(\pi_1(a+ib,c)\phi)(x)
=e^{i\lambda(c-bx+\frac{1}{2}ab)}f(x-a).\end{equation*}

We consider the Stratonovich-Weyl quantizer 
$\Omega_1$ on ${\mathbb R}^{2n}$ given by 
\begin{equation}\label{eq:Om1} \Omega_0(g \cdot (0,0)):=\pi_1(g)R_1\pi_1(g)^{-1}
\end{equation} for each $g\in G$.
Then it is well-known that the corresponding Stratonovich-Weyl correspondence $W_1$ defined on ${\mathbb R}^{2n}$ by 
\begin{equation*}W_1(A)(p,q):=\Tr (A\Omega_1(p,q))\end{equation*}
is exactly the inverse map of the classical Weyl correspondence $W^1$ which is 
usually defined by the formula
\begin{equation*}(W^1(F)\phi)
(p)={(2\pi)}^{-n}\,\int_{{\mathbb R}^{2n}}\, e^{isq} \,F( p+\tfrac{1}{2}s,
q)\,\phi (p+s)\,ds\,dq,
\end{equation*} see, for instance, \cite{AE}, \cite{GB}, \cite{CaGP}.

The connection between $W_0$ and $W_1$ (hence $W^1$) is given by the
Bargmann transform which is the unitary operator $B:L^2({\mathbb R}^{n})\rightarrow {\mathcal H}_0$ defined by
\begin{equation*}(B\phi)(z)=(\lambda/ \pi)^{n/4}\,\int_{{\mathbb R}^n}\,e^{
(1/4\lambda)z^2+ixz-(\lambda/2)x^2}\,\phi(x)\,dx.\end{equation*}

Indeed, it is well-known that $B$ intertwines $\pi_0$ and $\pi_1$, that is, we have for each $g\in G$, $\pi_1(g)=B^{-1}\pi_0(g)B$ \cite{Fo}, \cite{Comb}. Moreover, we can verify that $R_1=B^{-1}R_0B$. 
Taking Eq. \ref{eq:Om0} and  Eq. \ref{eq:Om1} into account, this implies that, for each trace-class operator $A$ on ${\mathcal H}_0$
and each $(p,q)\in {\mathbb R}^{2n}$,
we have
\begin{equation} \label{eq:connect} W_0(A)(q-i\lambda p)=W_1(B^{-1}AB)(p,q).\end{equation}

\begin{proposition} $W_0$ is a Stratonovich-Weyl correspondence for $(G_0,\pi_0,{\mathbb C}^n)$.\end{proposition}

\begin{proof}  For each trace-class operator $A$ on ${\mathcal H}_0$ and each $z\in {\mathbb C}^n$, we have 
\begin{equation*}W_0(A^{\ast})(z)=\Tr(A^{\ast}\Omega_0(z))=\Tr(\Omega_0(z)A)^{\ast}=
\overline{\Tr(\Omega_0(\xi)A}=\overline{W_0(A)(z)}.\end{equation*} Moreover, $W_0$ is covariant by construction. Finally, by Eq. \ref{eq:connect}, $W_0$ is unitary since $W_1$ is, see \cite{Fo}. \end{proof}

In particular, since $W_0$ is unitary, we have that $W_0^{-1}=W_0^{\ast}$.
Then we get, for each $F\in L^2({\mathbb C}^n,\mu_{\lambda})$
\begin{equation*}W_0^{-1}(F)=\int_{{\mathbb C}^n}\Omega_0 (z)F(z)\,d\mu_{\lambda}(z).\end{equation*}
Such a formula is particularly suitable for extending the Weyl calculus to 
the symmetric domains, see \cite{AU1}. However, the generalized Weyl calculus is, in general, no longer unitary \cite{AU2}.

Let us also mention that that the unitary part in the polar decomposition of the Berezin correspondence $S_0$ introduced in Section \ref{sec:2} is $W_0$, see \cite{Luo}, Theorem 6 and, for a slightly different proof based
on covariance, \cite{CaRiv}.

On the other hand, it is also known that the classical Weyl correspondence
can be extended to differential operators on ${\mathbb R}^{n}$, see for instance \cite{Ho}, \cite{Vo}. Here, we can similarly extend $W_0$ to operators of the form $A=\sum_{p,q}a_{pq}z^p(\frac{\partial}{\partial z})^q$ by using the integral formula for $W_0(A)$ given in Proposition \ref{prop:intW0}. 

Here, we use the standard multi-index notation. If $p=(p_1,p_2,\ldots,p_n)\in {\mathbb N}^n$, we set $z^{p}=z_1^{p_1}z_2^{p_2}\dots z_n^{p_n}$, $\vert p\vert=p_1+
p_2+\dots+p_n$, $p!=p_1!p_2!\dots p_n!$. Also, we say that $p\leq q$
if $p_k\leq q_k$ for each $k=1,2,\ldots, n$ and, in this case, we denote
$\tbinom{q}{p}=\tfrac{q!}{p!(q-p)!}$.

\begin{proposition}\label{prop:diff0}
For each $p,q \in {\mathbb N}^n$, let $A_{pq}:=z^p(\frac{\partial}{\partial z})^q$. Then we have 
\begin{equation*}W_0(A_{pq})(z)=(2\lambda)^{-\vert q\vert} \sum_{k\leq p,q}(-\lambda)^{\vert k\vert}\frac{p!\,q!}{k!(p-k)!(q-k)!}z^{p-k}{\bar z}^{q-k}.\end{equation*}
\end{proposition}

\begin{proof} Let $p,q \in {\mathbb N}^n$. By differentiating, for each $f\in {\mathcal H}_0$, the reproducing property
\begin{equation*} f(z) =\int_{{\mathbb C}^n}f(w)
e^{z
{\bar w}/2\lambda}
e^{-\vert
w\vert^2/2\lambda}d\mu_{\lambda}(w)\end{equation*}
with respect to $z$, we get 
\begin{equation*} z^p \left(\frac{\partial}{\partial z}\right)^q f =
(2\lambda)^{-\vert q\vert} z^p\int_{{\mathbb C}^n}f(w)
{\bar w}^q e^{z
{\bar w}/2\lambda}
e^{-\vert
w\vert^2/2\lambda}d\mu_{\lambda}(w).\end{equation*}
Hence $A_{pq}$ has kernel
\begin{equation*}k_{A_{pq}}(z,w)=(2\lambda)^{-\vert q\vert}z^p{\bar w}^q e^{z
{\bar w}/2\lambda}.\end{equation*}

Consequently, by using the integral formula of Proposition \ref{prop:intW0}, we have that
\begin{align*}W_0(A_{pq})(z)&=2^n(2\lambda)^{-\vert q\vert}
\int_{{\mathbb C}^n}w^p(2{\bar z}-{\bar w})^q e^{-\vert w-z\vert^2/\lambda}d\mu_{\lambda}(w)\\
&=2^n(2\lambda)^{-\vert q\vert}
\int_{{\mathbb C}^n}(w+z)^p({\bar z}-{\bar w})^q e^{-\vert w\vert^2/\lambda}d\mu_{\lambda}(w).
\end{align*}

Now, by the binome formula, we can write
\begin{align*} (w+z)^p&=\sum_{k\leq p}\binom{p}{k}z^{p-k}w^k;\\
({\bar z}-{\bar w})^q&=\sum_{l\leq q}\binom{q}{l}(-{\bar w})^l {\bar z}^{q-l} \end{align*}
and we can remark that for each $k,l\in {\mathbb N}^n$, we have
\begin{equation*}\int_{{\mathbb C}^n}w^k{\bar w}^l e^{-\vert w\vert^2/\lambda}d\mu_{\lambda}(w)=2^{-n}\lambda^{\vert k\vert }k!\delta_{kl}.\end{equation*}
By replacing in the previous expression of $W_0(A_{pq})(z)$, we obtain the
desired result.
\end{proof}

By computing in particular $W_0(d\pi_{0}(X))$ for $X\in {\mathfrak g}_0$,
we can verify that
\begin{equation*}W_0(d\pi_{0}(X))(z)=i \langle \Phi_{\lambda}
(z),X\rangle. \end{equation*}
for each $X\in {\mathfrak g}_0$ and $z\in {\mathbb C}^n$. This implies
that $W_0$ is also an adapted Weyl correspondence in the sense of \cite{CaWQ}.

\section {Orbits and representations of Heisenberg motion groups} \label{sec:4}

We fix a closed subgroup $K$ of $U(n)$. Then $K$ acts on $G_0$ by $k\cdot (z,c)=(kz,c)$ and we can form the semidirect product $G:=G_0\rtimes K$ which is called a Heisenberg motion
group. The elements of $G$ can be written as $(z,c,k)$ where $z\in {\mathbb C}^n$, $c\in {\mathbb R}$ and $k\in
K$. The multiplication of $G$ is then given by
\begin{equation*}(z,c,k)\cdot (z',c',k')=(z+kz',c+c'+\tfrac{1}{2}\omega (z,kz'),kk').\end{equation*}

We denote by
${\mathfrak k}$ and ${\mathfrak g}$ the Lie algebras of $K$ and $G$. The Lie brackets of ${\mathfrak g}$ are given by
\begin{equation*}[(v,c,A),(v',c',A')]=(Av'-A'v,
\omega (v,v'),[A,A']).\end{equation*} 

We compute the adjoint and coadjoint actions of $G$. 
Consider $g=(z_0,c_0,k_0)\in G$ and
$X=(v,c,A)\in {\mathfrak g}$. We can verify that
\begin{align*}&\Ad(g)X=\frac{d}{dt}(g\exp(tX)g^{-1})\vert_{t=0}\\
&=\bigl( k_0v-(\Ad(k_0)A)z_0, c+\omega (z_0,k_0v)-\tfrac{1}{2}\omega
(z_0,(\Ad(k_0)A)z_0),\Ad(k_0)A\bigr).\end{align*}

Let us denote by $\xi=(u,d,\phi)_{\ast}$, where $u\in {\mathbb C}^{n}$,
$d\in {\mathbb R}$ and $\phi \in {\mathfrak k}^{\ast}$, the
element of ${\mathfrak g}^{\ast}$ defined by
\begin{equation*}\langle \xi,(v,c,A)\rangle =\omega (u,v)+dc+\langle
\phi, A\rangle. \end{equation*} 
Also, for $u,v\in {\mathbb
C}^{n}$, we denote by $v\times u$ the element of ${\mathfrak
k}^{\ast}$ defined by $\langle v\times u, A\rangle:=\omega
(u,Av)$ for $A\in {\mathfrak k}$.
Then, from the formula for the adjoint action, we deduce that for each $\xi=(u,d,\phi)_{\ast}\in {\mathfrak g}^{\ast}$ and
$g=(z_0,c_0,k_0)\in G$ we have
\begin{equation*} \Ad^{\ast}(g)\xi= \bigl(k_0u-dz_0,d,\Ad^{\ast}(k_0)\phi
+z_0\times (k_0u-\tfrac {1}{2}dz_0)\bigr)_{\ast}.\end{equation*} 

From this, we deduce that if a coadjoint orbit of $G$ contains a point $(u,d,\phi)_{\ast}$
with $d\not= 0$ then it also contains a point of the form $(0,d,\phi_0)_{\ast}$. Such an orbit is called
\textit {generic}.

Now we recall the definition of the \textit {generic} unitary representations
of $G$ which are associated with the generic coadjoint orbits of $G$ \cite{CaRimut2}. Note that these representations can be obtained as holomorphically induced representations by the general method of \cite{Ne},
since $G$ is a quasi-Hermitian Lie group \cite{CaRimut1}.

Let $\rho$ be a unitary irreducible representation of $K$ on a (finite-dimensional) Hilbert space $V$ and let $\lambda >0$. Let 
$\mathcal H$ the Hilbert space of all holomorphic functions $f: {\mathbb C}^n\rightarrow V$ such that
\begin{equation*}\Vert f\Vert_{\mathcal H}^2
:=\int_{{\mathbb C}^n} \Vert f(z)\Vert_V^2\, e^{-\vert
z\vert^2/2\lambda}\, d\mu_{\lambda}(z)< +\infty.\end{equation*}

Then we can consider the representation $\pi$ of $G$ on $\mathcal H$
defined by
\begin{equation*}({\pi}(g)f)(z)=\exp \left(i\lambda c_0+\tfrac{1}{2}i{\bar z_0}z-\tfrac{\lambda}{4}\vert z_0\vert^2\right)\,\rho(k)\,f(k^{-1}(z+i\lambda z_0)) \end{equation*}
for each $g=(z_0, c_0,k)\in G$ and $z\in {\mathbb C}^n$.

Clearly, we have that ${\mathcal H}={\mathcal H}_0\otimes V$. For $f_0\in {\mathcal H}_0$ and $v\in V$,
we denote by $f_0\otimes v$ the function $z\rightarrow f_0(z)v$. Moreover, if $A_0$ is an operator on 
${\mathcal H}_0$ and $A_1$ is an operator on $V$ then we denote by $A_0\otimes A_1$ the operator
on ${\mathcal H}$ defined by $(A_0\otimes A_1)(f_0\otimes v)=A_0f_0\otimes A_1v$.

Also, we denote by $\tau$ the left-regular representation of $K$ on 
${\mathcal F}_0$, that is, one has $(\tau (k)f_0)(z)=f_0(k^{-1}z)$. Then we have 
\begin{equation}\label{eq:decomppi} \pi(z_0,c_0,k)=\pi_0(z_0,c_0) \tau(k)\otimes \rho(k)\end{equation} for each $z_0\in {\mathbb C}^n$,
$c_0\in {\mathbb R}$ and $k\in K$. This is precisely Formula (3.18) in \cite{BJLR}.

We can easily compute the differential of $\pi$: 

\begin{proposition}\label{redpi}  Let $X=(v,c,A)\in {\mathfrak g}$. Then, for each $f\in {\mathcal H}$ and each $z\in {\mathbb C}^n$, we have
\begin{equation*}(d\pi (X)f)(z)=d\rho(A)f(z)+i(\lambda c+\tfrac{1}{2}{\bar v}z)f(z)+df_z(-Az+i\lambda v)\end{equation*}
or, equivalently,
\begin{equation*}d\pi(X)=(d\pi_0(v,c)+d\tau (A))\otimes \id_V+\id_{{\mathcal H}_0}\otimes d\rho(A).\end{equation*}
\end{proposition}

In the rest of the paper, we fix $\phi_0$ such that $\rho$ is associated with the orbit $o(\phi_0)$ for the coadjoint action of $K$ as in \cite{CaRimut1},
\cite{Wild2}. Then the orbit $ {\mathcal O}(\xi_0)$ of $\xi_0=(0,\lambda, \phi_0)_{\ast}\in {\mathfrak g}^{\ast}$ for the coadjoint action of $G$ is associated with
$\pi$, see \cite{CaRimut2}.

\section{Berezin quantization for $\pi$} \label{sec:5}

In this section, we show that the Berezin calculus on $\mathcal H$ naturally
provides a diffeomorphism $\Psi:{\mathbb C}^n\times o(\phi_0)\rightarrow
{\mathcal O}(\xi_0)$. We begin with the Berezin calculus on $o(\phi_0)$.

The Berezin
calculus on $o(\phi_0)$ associates with each operator $A_1$ on $V$ a complex-valued
function $s_1(A_1)$ on the orbit $o(\phi_0)$ which is called the symbol of the
operator $A_1$ (see \cite{Be1}). We denote by $Sy(o({\phi_0}))$ the space of all such symbols. 

The following proposition summarizes some well-known properties of the
Berezin calculus, see for instance \cite{ACGc}, \cite{CGR}, \cite{CaS} and
\cite{Wild2}. 
\begin{proposition} \label{proprpetitber}
\begin{enumerate}
\item The map $A_1\rightarrow s_1(A_1)$ is injective.

\item We have $s_1(\id_{V})=1$.

\item  For each operator $A_1$ on $V$, we have $s_1(A_1^{\ast})=\overline
{s_1(A_1)}$.

\item  For each
operator $A_1$ on $V$, $k\in K$ and $\phi \in o(\phi_0)$,
we have
\begin{equation*} s_1(A_1)(\Ad^{\ast}(k)\phi)=s_1(\rho (k)^{-1}A_1 \rho
(k))(\phi).\end{equation*}
\end{enumerate}\end{proposition}

The Berezin calculus $S$ for $\pi$ can be defined as follows, \cite{CaRimut2}, \cite{CaBW}. For each operator $A_0$ on ${\mathcal H}_0$
and each operator $A_1$ on $V$, we set
\begin{equation*} S(A_0\otimes A_1):=S_0(A_0)\otimes s_1(A_1)
\end{equation*}
and then we extend $S$ by linearity to operators on $\mathcal H$.

Consider the action of $G$ on 
${\mathbb C}^n\times o(\varphi_0)$ defined by
\begin{equation*} g\cdot (z, \phi):=(kz-i\lambda z_0,\Ad^{\ast}(k)\phi)\end{equation*}
for $g=(z_0,c_0,k)\in G$, $z\in {\mathbb C}^n$ and $\phi \in o(\phi_0)$.
Then we can show that $S$ is $G$-covariant with respect to $\pi$ \cite{CaRimut2}.
Moreover, we have the following result, see \cite{CaRimut1}, \cite{CaRimut2}.

\begin{proposition} \label{symbpi} \begin{enumerate}
\item For each $X=(v,c,A)\in {\mathfrak g}$,
$z\in {\mathbb C}^n$ and 
$\phi\in o(\phi_0)$,  we have
\begin{equation*} S(d\pi (X))(z,\phi)=i\lambda c+\frac{i}{2}\left( {\bar v}z+ v{\bar z}\right)
-\frac{1}{2\lambda}{\bar z}(Az)
+s(d\rho(A))(\varphi).\end{equation*}
\item For each $X\in {\mathfrak g}$,
$z\in {\mathbb C}^n$ and 
$\phi\in o(\phi_0)$,  we have
\begin{equation*} S(d\pi (X))(z,\phi)=i\langle \Psi (z,\phi), X \rangle
\end{equation*}  where the map $\Psi: {\mathbb C}^n\times o(\phi_0)\rightarrow
{\mathfrak g}^{\ast}$ is defined by
\begin{equation*}\Psi (z,\phi)=\left(-iz, \lambda, \phi-\tfrac{1}{2\lambda}z
\times z\right)_{\ast}.\end{equation*}
\item $\Psi$ is a diffeomorphism from ${\mathbb C}^n\times o(\phi_0)$ onto ${\mathcal O}(\xi_0)$.

\item We have, for each $g\in G$, $z\in {\mathbb C}^n$ and 
$\phi\in o(\phi_0)$,
\begin{equation*}\Psi(g\cdot (z, \phi))=\Ad^{\ast}(g)\Psi (z, \phi).
\end{equation*} 
\end{enumerate} \end{proposition}

\section{Stratonovich-Weyl correspondence for $\pi$}\label{sec:6}

Here we construct a Stratonovich-Weyl correspondence for $\pi$ by combining the Weyl calculus on ${\mathbb C}^n$ and a
Stratonovich-Weyl correspondence for $\rho$.

We fix an invariant measure $\nu$ on $o(\phi_0)$ and consider $Sy(o(\varphi_0))$
as a (finite dimensional) subspace of $L^2(o(\phi_0),\nu)$.
On the other hand, we can equip $\End (V)$ with the Hilbert-Schmid norm.
Then
we can consider the unitary part $w_1$ in the polar decomposition of $s_1: \End (V)\rightarrow Sy(o(\phi_0))$. We immediately see that $w_1$ inherits some properties from $s_1$ and that $w_1$ is a 
Stratonovich-Weyl correspondence for $(K, \rho, o(\phi_0))$ \cite{FGBV},
\cite{CaPad}. Moreover, for each $\phi \in o(\phi_0)$, there exists a unique
$\omega_1(\phi)\in \End (V)$ such that
\begin{equation*}\label{eq:quantw
1} w_1(A_1)(\varphi)=\Tr(A_1\omega_1(\phi))\end{equation*}
for each $A_1\in \End (V)$.

Recall that such a map $\phi \rightarrow \omega_1(\phi)$ is called a  Stratonovich-Weyl quantizer and that the properties of $w_1$ are reflected by similar properties of $\omega_1$, see for instance \cite{GB}. In particular, the covariance property of $w_1$
is equivalent to the fact that for each $k\in K$ and $\phi\in o(\phi_0)$,
we have
\begin{equation*}\label{eq:covomega1} \omega_1(\Ad^{\ast}(k) \phi)=\rho(k)\omega_1(\phi)\rho(k)^{-1}.\end{equation*}

In the rest of this paper, we fix a section (defined on a dense open subset of $o(\phi_0)$) $\phi \rightarrow k_{\phi}$ for the action of $K$ on $o(\phi_0)$, see \cite{CaPar}. Then we have
\begin{equation*}\label{eq:defomega1} \omega_1( \phi)=\rho(k_{\phi})\omega_1(\phi_0)\rho(k_{\phi})^{-1}.\end{equation*}

By analogy with the Berezin calculus $S$, we define for each (suitable) operator $A_0$
on ${\mathcal  H}_0$ and each operator $A_1$ on $V$, the function
$W(A_0\otimes A_1)$ on ${\mathbb C}^n\times o(\phi_0)$ by
\begin{equation*}W(A_0\otimes A_1)(z,\phi):=W(A_0)(z)w_1(A_1)(\phi).
\end{equation*}
By Section \ref{sec:3}, this definition works in particular when $A_0$ is trace-class (or more generally Hilbert-Schmidt) and also when $A_0$ is a differential operator with polynomial coefficients. Of course, we can extend $W$ to finite sums of  operators of the form $A_0\otimes A_1$. In order to prove that $W$
is a Stratonovich-Weyl correspondence, we need the following lemma.

\begin{lemma} \label{lem:W0tau}
For each trace-class operator $A_0$ on ${\mathcal  H}_0$, $k\in K$ and
$z\in {\mathbb C}^n$ , we have 
\begin{equation*}W_0(\tau (k)^{-1}A_0\tau (k))(z)=W(A_0)(kz).
\end{equation*}\end{lemma}

\begin{proof} For each trace-class operator $A_0$ on ${\mathcal  H}_0$ and each $k\in K$, we can verify that the kernel of $\tau (k)^{-1}A_0\tau (k)$ is $(z,w)\rightarrow k_A(kz,kw)$. The result then follows from Proposition \ref{prop:intW0}. \end{proof}

Now we have the following result.

\begin{proposition}\label{prop:Wswc} W is a Stratonovich-Weyl correspondence for $(G,\pi, {\mathbb C}^n\times o(\phi_0))$. \end{proposition}

\begin{proof} From the properties of $W_0$ and $w_1$, we see immediately
that $W$ is unitary and satisfies the property that $W(A^{\ast})=\overline {W(A)}$ for $A$ operator on $\mathcal H$. It remains to prove that $W$ is
covariant with respect to $\pi$. We have just to consider the case where
$A=A_0\otimes A_1$ with $A_0$ operator on $\mathcal H$ and $A_1$ operator on $V$. Let $g=(z_0,c_0,k)\in G$. By using Eq. \ref{eq:decomppi}, we see that
\begin{equation*} W(\pi (g)^{-1}A\pi (g))=W_0(\tau(k)^{-1}\pi_0(z_0,c_0)^{-1}A_0\pi_0(z_0,c_0)\tau(k))\otimes
w_1(\rho(k)^{-1}A_1\rho(k)).\end{equation*}

Then, by using Lemma \ref{lem:W0tau} and the covariance of $W_0$ and $w_1$, we have, for each $(z,\phi)\in {\mathbb C}^n\times o(\phi_0)$,
\begin{align*}W(\pi (g)^{-1}A\pi (g))&(z,\phi)=
W_0(\pi_0(z_0,c_0)^{-1}A_0\pi_0(z_0,c_0))(kz)w_1(A_1)(\Ad^{\ast}(k)\phi)\\&=W_0(A_0)((z_0,c_0)\cdot kz)w_1(A_1)(\Ad^{\ast}(k)\phi)=
W(A)(g\cdot (z,\phi)).
\end{align*}
This ends the proof.
\end{proof}

Now, we aim to identify the Stratonovich-Weyl quantizer associated with $W$. We need the section for the action of $G$ on ${\mathbb C}^n\times o(\phi_0)$ defined by
\begin{equation*}(z,\phi)\rightarrow g_{(z,\phi)}:=({\lambda}^{-1}iz, 0,k_{\phi})\in G.\end{equation*} Recall that $\phi\rightarrow k_{\phi }$ is
a section for the action of $K$ on $o(\phi_0)$ and that $z\rightarrow g_z=
({\lambda}^{-1}iz, 0)\in G_0$ is a section for the action of $G_0$ on 
${\mathbb C}^n$.

Recall also that $R_0$ denotes the parity operator on ${\mathcal H}_0$,
that is, $(R_0f_0)(z)=2^nf_0(-z)$. Then we introduce the operator $R$
on $\mathcal H$ defined by $R:=R_0\otimes \omega_1(\phi_0)$ and
we define the Stratonovich-Weyl quantizer $\Omega$ by
\begin{equation*}\Omega(g\cdot (0,\phi_0)):=\pi(g)R\pi(g)^{-1}.\end{equation*}
This definition makes sense since $R$ commutes with $\pi(g)$ for each $g$
in the stabilizer $G_{(0,\phi_0)}$ of $(0,\phi_0)$ in $G$, that is, for each $g$ of the form
$(0,c_0,k)$ with $c_0\in {\mathbb R}$ and $k$ in the stabilizer of $\phi_0$ in $K$. Note that $G_{(0,\phi_0)}$ is also the stabilizer of $\xi_0$ for the coadjoint action of $G$.

In particular, we have, for each $(z,\phi)\in {\mathbb C}^n\times o(\phi_0)$,
\begin{equation*}\Omega(z,\phi)=\pi(g_{(z,\phi)})R\pi(g_{(z,\phi)})^{-1}.\end{equation*}

\begin{proposition} \label{prop:quantW} \begin{enumerate}
\item For each $(z,\phi)\in {\mathbb C}^n\times o(\phi_0)$, we have
\begin{equation*}\Omega(z,\phi)=\Omega_0(z)\otimes \omega_1(\phi);\end{equation*}
\item $\Omega$ is the quantizer associated with $W$, that is, for each trace-class operator $A$ on $\mathcal H$ and each $(z,\phi)\in {\mathbb C}^n\times o(\phi_0)$, we have 
\begin{equation*}W(A)(z,\phi)=\Tr(A\Omega(z,\phi)).\end{equation*}
\end{enumerate}\end{proposition}

\begin{proof} (1) First, note that for each $k\in K$, we have
\begin{equation*}\tau(k)^{-1}\Omega_0(0)\tau(k)=\Omega_0(0).\end{equation*}

Now, let $(z,\phi)\in {\mathbb C}^n\times o(\phi_0)$. We have
\begin{align*}\Omega (z,\phi)=&\pi(g_{(z,\phi)})R\pi(g_{(z,\phi)})^{-1}\\
=&(\pi_0(g_z)\tau(k_{\phi})\otimes \rho(k_{\phi}))(\Omega_0(0)\otimes
\omega_1(\phi_0))(\tau(k_{\phi})^{-1}\pi_0(g_z)^{-1}\otimes \rho(k_{\phi})^{-1})\\
=&(\pi_0(g_z)\tau(k_{\phi})\Omega_0(0)\tau(k_{\phi})^{-1}\pi_0(g_z)^{-1})
\otimes (\rho(k_{\phi})\omega_1(\phi_0)\rho(k_{\phi})^{-1})\\
=&(\pi_0(g_z)\Omega_0(0)\pi_0(g_z)^{-1})\otimes \omega_1(\phi)\\
=&\Omega_0(z)\otimes \omega_1(\phi).
\end{align*}

(2) We have just to consider the case where $A=A_0\otimes A_1$ where $A_0$ is a trace-class operator on ${\mathcal H}_0$. Then we have

\begin{align*}
W(A)(z,\phi)=& \Tr(A_0\otimes A_1)\Tr(\Omega_0(z)\otimes \omega_1(\phi))\\
=&\Tr(A_0\Omega_0(z)\otimes A_1\omega_1(\phi))\\
=&\Tr(A_0\Omega_0(z))\Tr(A_1\omega_1(\phi))\\
=&W_0(A_0)(z)w_1(A_1)(\phi)\\
=&W(A)(z,\phi).
\end{align*} This ends the proof.
\end{proof}

Note that Proposition \ref{prop:quantW} also gives a construction of the
Stratonovich-Weyl correspondence for $\pi$ in the spirit of the general method
of \cite{GB} and \cite{AMO}. We can also give an integral expression of $W(A)$ and then extend Proposition \ref{prop:intW0}. Note that each
Hilbert-Schmidt operator $A$ on $\mathcal H$ has a kernel $K_A:(z,w)\rightarrow K_A(z,w)\in \End(V)$ so that, for each $f\in {\mathcal H}$
and each $z \in {\mathbb C}^n$, we have
\begin{equation*}(Af)(z)=
\int_{{\mathbb C}^n}K_A(z,w)f(w)\,e^{-\vert
w\vert^2/2\lambda}d\mu_{\lambda}(w).
\end{equation*} 

\begin{proposition}  \label{prop:intW}
For each trace-class operator $A$ on ${\mathcal H}$
and each $(z,\phi)\in {\mathbb C}^n\times o(\phi_0)$, we have
\begin{equation*}
W(A)(z,\phi)=2^n\int_{{\mathbb C}^n}w_1(K_A(w,2z-w))(\phi)\exp \left(\tfrac{1}{\lambda}\left(-z{\bar z}+z{\bar w}-\tfrac{1}{2}w{\bar w}\right)\right) d\mu_{\lambda}(w). 
\end{equation*}
\end{proposition}

\begin{proof} As usual, it is sufficient to consider the case where $A=A_0\otimes A_1$ where $A_0$ is a trace-class operator on ${\mathcal H}_0$ and $A_1$ an operator on $V$. Clearly, we have $K_A(z,w)=k_{A_0}(z,w)A_1$ and then for each
$(z,\phi)\in {\mathbb C}^n\times o(\phi_0)$ we have
\begin{equation*}W(A)(z,\phi)=W_0(A_0)(z)w_1(A_1)(\phi).\end{equation*}
Hence the result follows from Proposition \ref{prop:intW0}.
\end{proof}

Note that, since $W_0$ is the unitary part in the polar decomposition of $S_0:{\mathcal L}_2({\mathcal H}_0)\rightarrow L^2({\mathbb C}^n,\mu_{\lambda})$, see Section \ref{sec:3}, we have that $W$ is
the unitary part in the polar decomposition of $S:{\mathcal L}_2({\mathcal H})\rightarrow L^2({\mathbb C}^n,\mu_{\lambda})\otimes V$ .
It is also clear that since the Bargmann transform connects $W_0$ to $W_1$
(see Section \ref{sec:3}), it also connects $W$ to the Stratonovich-Weyl correspondence constructed in \cite{CaRimut2}.

We can immediately deduce a Stratonovich-Weyl correspondence for $(G,\pi, {\mathcal O}(\xi_0))$ from $W$. Indeed, let ${\mathcal W}$ be the map
defined by ${\mathcal W}(A)=W(A)\circ \Psi^{-1}$ for each trace-class operator $A$ on $\mathcal H$. Then, as an immediate consequence of 
Proposition \ref{prop:Wswc}, we have the following result.

\begin{proposition} ${\mathcal W}$ is a Stratonovich-Weyl correspondence for $(G,\pi, {\mathcal O}(\xi_0))$. \end{proposition}

We finish by giving an explicit expression for $W(d\pi(X)$, $X\in {\mathfrak g}$ in the spirit of Proposition \ref{symbpi}.

\begin{proposition} For each $X=(v,c,A)\in {\mathfrak g}$ and each
$(z,\phi)\in {\mathbb C}^n\times o(\phi_0)$, we have 
\begin{equation*} W(d\pi (X))(z,\phi)=i\lambda c+\frac{i}{2}\left( {\bar v}z+ v{\bar z}\right)
-\frac{1}{2\lambda}{\bar z}(Az)+\frac{1}{2\lambda}\Tr(A)
+w_1(d\rho(A))(\varphi).\end{equation*}
\end{proposition}

\begin{proof} Let $X=(v,c,A)\in {\mathfrak g}$. Recall that (Proposition
\ref{redpi})
\begin{equation*}d\pi(X)=(d\pi_0(v,c)+d\tau (A))\otimes \id_V+\id_{{\mathcal H}_0}\otimes d\rho(A).\end{equation*}
This implies that
\begin{equation}\label{eq:Wdpi} W(d\pi(X))(z,\phi)=W_0(d\pi_0(v,c))(z)+
W_0(d\tau(A))(z)+w_1(d\rho(A))(\phi).\end{equation}

But by Proposition \ref{prop:diff0} we have
\begin{equation*}W_0(d\pi_0(v,c))(z)=i\lambda c+\frac{i}{2}\left( {\bar v}z+ v{\bar z}\right).\end{equation*}
Moreover, writing $A=(a_{ij})$, we get 
\begin{equation*}(d\tau(A)f_0)(z)=-df_0(z)(Az)=-\sum_{ij}a_{ij}z_j\frac{\partial f_0}{\partial z_i}
\end{equation*} for each $f_0\in {\mathcal H}$. Then, applying
Proposition \ref{prop:diff0} again we get
\begin{align*}W_0(d\tau(A))(z)&=-\frac{1}{2\lambda}\sum_{ij}a_{ij}{\bar z}_iz_j+\sum_ia_{ii}\\&=-\frac{1}{2\lambda}{\bar z}(Az)+\frac{1}{2\lambda}\Tr(A).
\end{align*}

Hence, by replacing in Eq. \ref{eq:Wdpi}, we obtain the desired result.
\end{proof}

Note that if $K\subset SU(n)$, one has $\Tr(A)=0$ for each $A\in {\mathfrak k} $ and the preceding formula for $W_0(d\pi(X))$ is very close to the formula
for $S(d\pi(X))$, see Proposition \ref{symbpi}.

\end{document}